\documentclass[12pt]{amsart}

\usepackage{amsmath}
\usepackage{amsfonts}
\usepackage{amsthm}

\usepackage[cmtip,arrow]{xy}
\usepackage{pb-diagram, pb-xy}

\newtheorem{corollary}{Corollary}
\newtheorem{lemma}{Lemma}
\newtheorem{proposition}{Proposition}
\newtheorem{definition}{Definition}

\renewcommand{\P}{\mathcal{P}}

\begin{document}

\title{Gorenstein flat and projective (pre)covers}
\author{S. Estrada, A. Iacob, S. Odabasi}
\thanks{key words: Gorenstein flat module, Gorenstein flat complex, Gorenstein flat cover, Gorenstein projective module, Gorenstein projective complex, Gorenstein projective precover}

\maketitle

\begin{abstract}
We consider a right coherent ring $R$. We prove that the class of Gorenstein flat complexes is covering in the category of complexes of left $R$-modules $Ch(R)$.\\
When $R$ is also left n-perfect, we prove that the class of Gorenstein projective complexes is special precovering in $Ch(R)$.
\end{abstract}

\section{introduction}

Gorenstein homological algebra is the relative version of homological algebra that uses Gorenstein projective, Gorenstein flat and Gorenstein injective modules instead of the classical projective, flat, and injective modules. But while the classical projective and injective resolutions are known to exist over arbitrary rings, things are different when it comes to the existence of the Gorenstein projective and Gorenstein injective resolutions. Their existence is well known over Gorenstein rings. But for arbitrary rings this is still an open question.\\
We focus on the existence of the Gorenstein projective precovers. So far the existence of Gorenstein projective precovers is known over commutative noetherian rings of finite Krull dimension (\cite{murfet-salarian:11: totally.acyclic}). In fact, over such rings, the class of Gorenstein projective modules is special precovering (\cite{EEI}). We prove (Proposition 6) that if $R$ is a right coherent ring that is also left n-perfect then the class of Gorenstein projective complexes is special precovering in the category $Ch(R)$. In particular, every left $R$-module $M$ has a special Gorenstein projective precover in this case. Examples of such rings include but are not limited to: Gorenstein rings, commutative noetherian rings of finite Krull dimension, as well as two sided noetherian rings $R$ such that $inj.dim_R R < \infty$.\\
We also consider the question of the existence of the Gorenstein flat covers for complexes. It has been proved recently (\cite{yang:14: Gorenstein.flat.precovers}) that the class of Gorenstein flat modules is precovering over any asoociative ring with unity. However when it comes to complexes of modules, the best result is the following: If $R$ is a two sided noetherian ring then every complex of $R$-modules has a Gorenstein flat cover (\cite{EEI}). We show here (Proposition 4) that the class of Gorenstein flat complexes is covering over any right coherent ring $R$.

\section{preliminaries}

Throughout this section $R$ denotes an associative ring with unity.\\

\begin{definition} (\cite{enochs:00:relative}, definition 10.2.1)
An $R$-module $M$ is Gorenstein projective if there exists an exact and $Hom(-,Proj)$ exact sequence of projective $R$-modules $$\ldots \rightarrow P_1 \rightarrow P_0 \rightarrow P_{-1} \rightarrow \ldots $$ such that $M=Ker(P_0 \rightarrow P_{-1})$.
\end{definition}

We will use the notation $\mathcal{GP}$ for the class of Gorenstein projective modules.\\

By replacing "projective module" with "projective complex" in the above definition, we obtain the definition of a Gorenstein projective complex in the category of complexes of left $R$-modules, $Ch(R)$. We recall that a complex $P$ is projective if $P$ is exact and if each $Z_n(P)$ is a projective module.\\
It is known (\cite{liu:11: Gorenstein.complexes}) that over any ring $R$ a complex $C$ is Gorenstein projective if and only if each $C_n$ is a Gorenstein projective $R$-module.

The Gorenstein flat modules are defined in terms of the tensor product.\\
\begin{definition} (\cite{enochs:00:relative}, Definition 10.3.1)
A left $R$-module $G$ is Gorenstein flat if there exists an exact and $Inj \otimes -$ exact sequence of flat left $R$-modules $$\ldots \rightarrow F_1 \rightarrow F_0 \rightarrow F_{-1} \rightarrow \ldots $$ such that $G=Ker(F_0 \rightarrow F_{-1})$.
\end{definition}

We will use the notation $\mathcal{GF}$ for the class of Gorenstein flat modules.\\

The Gorenstein flat complexes were defined in \cite{garcia:99:covers}.
We recall that if $C$ is a complex of right $R$-modules and $D$ is a
complex of left $R$-modules then the usual tensor product complex of
$C$ and $D$ is the complex of $Z$-modules $C \otimes ^. D$ with $(C
\otimes ^. D)_n = \oplus_{t \in Z} (C_t \otimes_R D_{n-t})$ and
differentials $$\delta (x \otimes y) = \delta^C _t (x) \otimes y +
(-1)^t x
\otimes \delta^D _{n-t}(y)$$ for $x \in C_t$ and $y \in D_{n-t}$.\\

In \cite{garcia:99:covers},  Garc\'{\i}a Rozas introduced another
tensor product: if $C$ is again a complex of right $R$-modules and
$D$ is a complex of left $R$-modules then $C \otimes D$ is defined
to be $\frac{C \otimes ^. D}{B(C \otimes ^. D)}$. Then with the maps
$$\frac{(C \otimes ^. D)_n}{B_n(C \otimes ^. D)} \rightarrow \frac{(C
\otimes ^. D)_{n-1}}{B_{n-1}(C \otimes ^. D)}$$ $x \otimes y
\rightarrow \delta_C (x) \otimes y$, where $x \otimes y$ is used to
denote the coset in $\frac{C \otimes ^. D}{B(C \otimes ^. D)}$ we
get a complex. This is the tensor product used to define Gorenstein flat complexes.\\
We recall that a compex $F$ is flat if $F$ is exact and $Z_n(F)$ is a flat module for any integer n.

\begin{definition} (\cite{garcia:99:covers}, Definition 5.4.1) A complex $G$ of left $R$-modules is Gorenstein flat if there exists an exact and $Inj \otimes -$ exact sequence of flat complexes (of left $R$-modules) $$\ldots \rightarrow F_1 \rightarrow F_0 \rightarrow F_{-1} \rightarrow \ldots $$ such that $G=Ker(F_0 \rightarrow F_{-1})$.
\end{definition}

It is known (\cite{liu:13: Gorenstein.complexes}) that over a right coherent ring $R$, a complex $G$ is Gorenstein flat if and only if each module $G_n$ is Gorenstein flat.

\section{Gorenstein flat covers for complexes}

We prove first that we have a Neeman's type result for Gorenstein flat modules.\\
We will use the following results:\\

\begin{proposition}(\cite{enochs:02:kaplansky}, Prop. 2.10)
For any ring $R$ the class of Gorenstein flat complexes is a Kaplansky class.
\end{proposition}

\begin{proposition} (\cite{enochs:00:relative}, Cor. 2.1.9)
If $R$ is a right coherent ring then the class of Gorenstein flat left $R$-modules is closed under direct limits.
\end{proposition}

We will use the notation $Ch(\mathcal{GF})$ for the class of complexes of Gorenstein flat left $R$-modules.\\

Our first result is:\\
\begin{proposition}
If $R$ is a right coherent ring then the inclusion functor $\textbf{K}(\mathcal{GF}) \rightarrow \textbf{K}(R)$ has a right adjoint.
\end{proposition}

\begin{proof}
By Propositions 1 and 2, the cotorsion pair $(\mathcal{GF}, \mathcal{GF}^\bot)$ is cogenerated by a set. Then by \cite{enochs:11:relative2}, Th. 7.2.14, the pair $(Ch(\mathcal{GF}), Ch(\mathcal{GF})^\bot)$ is a cotorsion pair cogenerated by a set. Then again by \cite{enochs:11:relative2}, Th. 5.1.7, the inclusion functor $\textbf{K}(\mathcal{GF}) \rightarrow \textbf{K}(R)$ has a right adjoint.
\end{proof}

We show that if $R$ is a right coherent ring then the class of Gorenstein flat complexes (of left $R$-modules) is covering in $Ch(R)$.\\

\begin{proposition}
Let $R$ be a right coherent ring. Then every complex of left $R$-modules has a Gorenstein flat cover.
\end{proposition}

\begin{proof}
Since the cotorsion pair $(Ch(\mathcal{GF}), Ch(\mathcal{GF})^\bot)$ is cogenerated by a set, it is a complete one. So the class $Ch(\mathcal{GF})$ is precovering. By Proposition 2, this class is also closed under direct limits, and therefore it is a covering class in $Ch(R)$. By \cite{liu:13: Gorenstein.complexes} Prop. 3.2, this is the class of Gorenstein flat complexes.
\end{proof}

\section{Gorenstein projective (pre)covers}

We recall that a ring $R$ is left n-perfect if for any flat left $R$-module $F$, $p.d._RF \le n$.\\
We also recall that the character module of a left $R$-module $M$ is the right $R$-module $M^+ = Hom_Z(M, Q/Z)$. Then $M^{++} = (M^+)^+$, for any $_RM$. In the following we use the notation $Flat^{++}$ for the class of all left $R$-modules of the form $C^{++}$, where $C$ is any flat left $R$-module.

We prove that the class of Gorenstein projective modules is special precovering over a right coherent ring $R$ that is left n-perfect.\\
We begin with the following:\\
\begin{lemma}
Let $R$ be a right coherent ring that is left n-perfect and let $F$ be a flat left $R$-module. Then there exists an exact sequence $$0 \rightarrow F \rightarrow S^0 \rightarrow S^1 \rightarrow \ldots \rightarrow S^{n-1} \rightarrow C \rightarrow 0$$ with all $S^i$ in $Flat^{++}$ and with $C$ pure injective and flat.
\end{lemma}

\begin{proof}
Since $R$ is right coherent we have that a module $F$ is flat if and only if $F^{++}$ is flat (\cite{cheatham:81:flat}, Theorem 1)\\
The sequence $0 \rightarrow F \rightarrow F^{++} \rightarrow \frac{F^{++}}{F} \rightarrow 0$ is pure exact with $F^{++}$ flat. Therefore the module $\frac{F^{++}}{F}$ is flat. Repeating we obtain an exact complex $$0 \rightarrow F \rightarrow S^0 \rightarrow S^1 \rightarrow S^2 \ldots $$ with each $S^i$ in $Flat^{++}$ and with each $C^i = Im (S^i \rightarrow S^{i+1})$ flat.\\
Let $K$ be any flat $R$-module. The exact sequence $0 \rightarrow F \rightarrow S^0 \rightarrow C^0 \rightarrow 0$ gives a long exact sequence $Ext^1(K, S^0) \rightarrow Ext^1(K, C^0) \rightarrow Ext^2(K,F) \rightarrow Ext^2(K,S^0) \rightarrow \ldots $.\\
Since $K$ is flat and $S^0$ is in $Flat^{++}$ hence pure injective, we have $Ext^i(K, S^0) =0$ for all $i \ge 1$. Therefore $Ext^j(K, C^0) \simeq Ext^{j+1} (K, F)$ for all $j \ge 1$.\\
Similarly $Ext^j(K, C^{n-2}) \simeq Ext^{j+n}(K, F)$ for all $j \ge 1$. But the ring $R$ is left n-perfect and $K$ is flat so $proj.dim (K) \le n$. Then $Ext^{j+n} (K,F) = 0$ for all $j \ge 1$.\\
 So $Ext^j(K, C^{n-2}) =0$ for all $j \ge 1$ and for any flat module $K$.\\
 In particular we have that $Ext^1 (C^{n-1}, C^{n-2}) = 0$. This means that the exact sequence $0 \rightarrow C^{n-2} \rightarrow S^{n-1} \rightarrow C^{n-1} \rightarrow 0$ is split exact. Thus $C^{n-1}$ is a direct summand of $S^{n-1} \in Flat^{++}$ and so $C = C^{n-1}$ is pure injective.
\end{proof}

\begin{lemma}
Let $R$ be a right coherent and left n-perfect ring and let $G$ be a (left) Gorenstein flat module. Then $Ext^i(G,F) = 0$ for any flat and cotorsion module $F$, for any $i \ge 1$.
\end{lemma}

\begin{proof}
Let $F$ be a flat and cotorsion $R$-module. By Lemma 1 there is an exact complex $0 \rightarrow F \rightarrow S^0 \rightarrow S^1 \rightarrow \ldots \rightarrow S^{n-1} \rightarrow C \rightarrow 0$ with $S^i$ in $Flat^{++}$ for all i and with $C$ a pure injective module. Also, by the proof of Lemma 1, $C$ is a direct summand of $S^{n-1}$.\\
The exact sequence $0 \rightarrow F \rightarrow S^0 \rightarrow C^0 \rightarrow 0$ gives a long exact sequence $Ext^1(G, S^0) \rightarrow Ext^1(G, C^0) \rightarrow Ext^2(G, F) \rightarrow Ext^2(G,S^0)$. \\
We have $S^0 = F^{++}$, so for any $i \ge 1$, $Ext^i(G,S^0) = Ext^i(G, F^{++}) \simeq Ext^i(F^+, G^+) = 0$ because $F^+$ is injective and $G^+$ is Gorenstein injective. Then by the above, we have that $Ext^1(G,C^0) \simeq Ext^2(G,F)$. The same argument gives that $Ext^i(G, C^0) \simeq Ext^{i+1} (G, F)$ for any $i \ge 1$.\\
Similarly, $Ext^i(G, C^1) \simeq Ext^{i+1}(G, C^0) \simeq Ext^{i+2}(G, F)$ for all $i \ge 1$. Continuing we obtain that $Ext^i(G, C^{n-1}) \simeq Ext^{i+1}(G, C^{n-2}) \simeq \ldots \simeq Ext^{i+n}(G, F)$, for any $i \ge 1$.\\
But $C = C^{n-1}$ is a direct summand of $S^{n-1} \in Flat^{++}$, and $Ext^i(G, S^{n-1}) =0$ for all i. It follows that $Ext^i(G, C)=0$ for all $i \ge 1$.\\
Thus $Ext^{i+n}(G, F)=0$ for any $i \ge 1$, for any flat $R$-module $F$ and for any Gorenstein flat $R$-module $G$.\\

Since $G$ is a Gorenstein flat $R$-module, there is an exact complex $0 \rightarrow G \rightarrow F^0 \rightarrow F^1 \rightarrow \ldots$ with each $F^l$ flat and with each $G^l = Im (F^l \rightarrow F^{l+1})$ Gorenstein flat. Also, since $F$ is cotorsion, we have $Ext^i(F^l, F)=0$ for all $l \ge 0$.\\
The short exact sequence $0 \rightarrow G \rightarrow F^0 \rightarrow G^0 \rightarrow 0$ gives an exact sequence $0 = Ext^1(F^0,F) \rightarrow Ext^1(G,F) \rightarrow Ext^2(G^0 , F) \rightarrow Ext^2(F^0, F) = 0$. So $Ext^1(G,F) \simeq Ext^2(G^0, F)$. Similarly $Ext^i (G,F) \simeq Ext^{i+1}(G^0, F)$ for all $i \ge 1$. \\ Continuing we obtain that $Ext^i(G,F) \simeq Ext^{i+n}(G^{n-1},F)$ for all $i \ge 1$. By the above, $Ext^{i+n}(G^{n-1},F)=0$ for all $i \ge 1$. Therefore $Ext^i(G,F) = 0$ for any flat and cotorsion module $F$, for all $i \ge 1$.

\end{proof}

We recall that an exact complex $C$ of flat $R$-modules is N-totally acyclic if $E \otimes C$ is still exact for any injective $R$-module $E$. In particular, if $C$ is N-totally acyclic then for each integer n, $Z_n(C)$ is Gorenstein flat.


\begin{lemma}
Let $R$ be a right coherent ring that is left n-perfect. If $C$ is an N-totally acyclic complex of projective modules, then $C$ is $Hom(-,Q)$ exact for any flat $R$-module $Q$.
\end{lemma}

\begin{proof}
Let $Q$ be any flat $R$-module. By Lemma 1 there is an exact complex $0 \rightarrow Q \rightarrow S^0 \rightarrow S^1 \rightarrow \ldots \rightarrow S^n \rightarrow 0$ with each $S^i$ pure injective and flat.\\
Let $M=Z_0(C)$. Since $M$ is Gorenstein flat we have that $Ext^j(M,K)=0$  for any flat and cotorsion module $K$ and any $j \ge 1$ (by Lemma 2). In particular, $Ext^j(M,S^i)=0$ for all $j \ge 1$, for all $0 \le i \le n$. So by the above, we have that $Ext^i(M,Q)=0$ for all $i \ge n+1$.\\
But if $T = Z_n(C)$ then since each module $C_i$ is projective we have that $Ext^j(M,Q) \simeq Ext^{j + n} (T,Q)$ for all $j \ge 1$. So we have that $Ext^i(T,Q)=0$ for all $i \ge 1$. But if we replace $M$ with $Z_{-n}(C)$ then a similar argument gives that $Ext^i(M,Q) =0$ for all $i \ge 1$. Similarly, $Ext^i(Z_j(C), Q)=0$, for all j and for all $i \ge 1$. So $Hom(C,Q)$ is an exact complex for any flat module $Q$.
\end{proof}

We recall that an exact complex $C$ of projective modules is called totally acyclic if $Hom(C,P)$ is exact for any projective module $P$.\\

\begin{corollary}
Let $R$ be a right coherent ring that is left n-perfect. Then any N-totally acyclic complex of projective modules is totally acyclic
\end{corollary}

\begin{proof}
By Lemma 3, for any N-totally acyclic complex of projective modules, $C$, the complex $Hom(C,P)$ is still exact, for any projective module $P$. It follows that $C$ is totally acyclic.
\end{proof}

\begin{proposition}
Let $R$ be a right coherent ring. If $R$ is left $n$-perfect then every Gorenstein flat $R$-module $M$ has Gorenstein projective dimension less than or equal to $n$.
\end{proposition}

\begin{proof}
Since $M$ is Gorenstein flat there is an exact and $Inj \otimes -$ exact complex $\overline{F} = \ldots \rightarrow F_1 \rightarrow F_0 \rightarrow F_{-1} \rightarrow \ldots $ such that $M = Ker (F_0 \rightarrow F_{-1})$. \\
Consider a partial projective resolution of $\overline{F} : 0 \rightarrow C \rightarrow P_{n-1} \rightarrow \ldots \rightarrow P_0 \rightarrow \overline{F} \rightarrow 0$. Since $\overline{F}$ is $Inj \otimes -$ exact and each $P_k$ is a projective complex, it follows that $C$ is exact and $Inj \otimes -$ exact.\\
For each $i$ we have an exact complex $0 \rightarrow C_i \rightarrow P_{i,n-1} \rightarrow P_{i, n-2} \rightarrow \ldots \rightarrow \P_{i,0} \rightarrow F_i \rightarrow 0$ with $P_{i,k} \in Proj$. Since the projective dimension of $F_i$ is less than or equal to n, it follows that each $C_i$ is projective. So $C$ is an exact and $Inj \otimes -$ exact complex of projective modules. By Lemma 3, $C$ is totally acyclic. Then $Z_j(C)$ is Gorenstein projective for each j. The exact sequence of exact complexes $0 \rightarrow C \rightarrow P_{n-1} \rightarrow \ldots \rightarrow P_0 \rightarrow \overline{F} \rightarrow 0$ gives an exact sequence of modules $ 0 \rightarrow Z_j(C) \rightarrow Z_j(P_{n-1}) \rightarrow \ldots \rightarrow Z_j(\overline{F}) \rightarrow 0$. By the above each $Z_j(C)$ is Gorenstein projective. Since each $P_i$ is a projective complex, it follows that $Z_j(P_i)$ is a projective module for all j. So for each j, $Z_j(\overline{F})$ has Gorenstein projective dimension less than or equal to n. In particular, $G.p.d. M \le n$.
\end{proof}

We can prove now the existence of special Gorenstein projective precovers over a right coherent and left perfect ring $R$.\\
We recall that (by \cite{christensen:06:gorenstein}, Proposition 3.7), if $R$ is right coherent and any flat $R$-module has finite projective dimension, then any Gorenstein projective module is also Gorenstein flat. Then by \cite{liu:11: Gorenstein.complexes} and by \cite{liu:13: Gorenstein.complexes}, over a right coherent ring that is also left n-perfect, every Gorenstein projective complex is also a Gorenstein flat complex.\\
We use the notation $GorProj$ for the class of Gorenstein projective complexes and we denote by $GorFlat$ the class of Gorenstein flat complexes.

\begin{proposition}
Let $R$ be a right coherent ring. If $R$ is left $n$-perfect then the class of Gorenstein projective complexes is special precovering.
\end{proposition}

\begin{proof}
- We show first that every Gorenstein flat complex $G$ has a special Gorenstein projective precover.\\

Let $$0 \rightarrow \overline{G} \rightarrow P_{n-1} \rightarrow \ldots \rightarrow P_0 \rightarrow G \rightarrow 0$$

be a partial projective resolution of $G$. Then for each $j$ we have an exact sequence of modules
$$0 \rightarrow \overline{G}_j \rightarrow P_{n-1,j} \rightarrow \ldots \rightarrow P_{0,j} \rightarrow G_j \rightarrow 0$$

Since $Gpd$ $G_j \le n$ (by Proposition 5) it follows that each $\overline{G}_j$ is Gorenstein projective. Thus $\overline{G}$ is a Gorenstein projective complex (by \cite{liu:11: Gorenstein.complexes}, Theorem 2.2). So there exists an exact and $Hom(-,Proj)$ exact complex of projective complexes $$0 \rightarrow \overline{G} \rightarrow T_{n-1} \rightarrow \ldots \rightarrow T_0 \rightarrow \ldots$$\\
Let $T=Ker(T_{-1} \rightarrow T_{-2})$. Then $T$ is a Gorenstein projective complex, and we have a commutative diagram:\\

\[
\begin{diagram}
\node{0}\arrow{e}\node{\overline{G}}\arrow{s,=}\arrow{e}\node{T_{n-1}}\arrow{s}\arrow{e}\node{\cdots}\arrow{e}\node{T_1}\arrow{s}\arrow{e}\node{T_0}\arrow{s}\arrow{e}\node{T}\arrow{s}\arrow{e}\node{0}\\
\node{0}\arrow{e}\node{\overline{G}}\arrow{e}\node{P_{n-1}}\arrow{e}\node{\cdots}\arrow{e}\node{P_1}\arrow{e}\node{P_0}\arrow{e}\node{G}\arrow{e}\node{0}
\end{diagram}
\]

Therefore we have an exact sequence:\\
$$0 \rightarrow T_{n-1} \rightarrow P_{n-1} \oplus T_{n-2} \rightarrow \ldots \rightarrow P_1 \oplus T_0 \rightarrow P_0 \oplus T \xrightarrow{\delta} G \rightarrow 0$$

Let $V=Ker{\delta}$. Then $V$ has finite projective dimension, so $Ext^1(W,V)=0$
for any Gorenstein projective complex $W$.\\
We have an exact sequence $0 \rightarrow V \rightarrow P_0 \oplus T \rightarrow G \rightarrow 0$ with $P_0 \oplus T$ Gorenstein projective and with $V$ of finite projective dimension. Thus $P_0 \oplus T \rightarrow G$ is a special Gorenstein projective precover.\\

- We prove now that every complex $X$ has a special Gorenstein projective precover.\\

Let $X$ be any complex of $R$-modules. 
By Proposition 4, every complex over a right coherent ring has a Gorenstein flat cover. So there exists an exact sequence $$0 \rightarrow Y \rightarrow G \rightarrow X \rightarrow 0$$ with $G$ Gorenstein flat and with $Ext^1(U,Y)=0$ for any Gorenstein flat complex $U$.\\

By the above, there is an exact sequence $$0 \rightarrow L \rightarrow P \rightarrow G \rightarrow 0$$ with $P$ Gorenstein projective and with $L$ complex of finite projective dimension.\\
Form the pullback diagram

\[
\begin{diagram}
\node{}\node{L}\arrow{s}\arrow{e,=}\node{L}\arrow{s}\\
\node{0}\arrow{e}\node{M}\arrow{s}\arrow{e}\node{P}\arrow{s}\arrow{e}\node{X}\arrow{s,=}\arrow{e}\node{0}\\
\node{0}\arrow{e}\node{Y}\arrow{e}\node{G}\arrow{e}\node{X}\arrow{e}\node{0}
\end{diagram}
\]

Since $L \in {GorProj}^{\perp}$, $Y \in GorFlat^{\perp}$, $GorProj \subseteq GorFlat$, and the sequence $0 \rightarrow L \rightarrow M \rightarrow Y \rightarrow 0$ is exact, it follows that $M \in {GorProj}^{\perp}$.\\
So $0 \rightarrow M \rightarrow P \rightarrow X \rightarrow 0$ is exact with $P$ Gorenstein projective and with $M \in {GorProj}^{\perp}$

\end{proof}

\begin{corollary}
If $R$ is a right coherent ring that is left n-perfect then every module has a special Gorenstein projective precover.
\end{corollary}

\begin{proof}
Consider a left $R$-module $M$ and let $\underline{M}$ denote the complex with $M$ in the zeroth place and zeros everywhere else. By Proposition 6 there exists an exact sequence $0 \rightarrow L \rightarrow P \rightarrow M \rightarrow 0$  with $P$ a Gorenstein projective complex, and with $L \in GorProj^\bot$. In particular there is an exact sequence of modules $0 \rightarrow L_0 \rightarrow P_0 \rightarrow M \rightarrow 0$ with $P_0$ a Gorenstein projective module (by [13]).\\
We show that $L_0 \in \mathcal{GP}^\bot$.
Let $G$ be any Gorenstein projective left $R$-module and let $\overline{G} = \ldots \rightarrow 0 \rightarrow G \rightarrow G \rightarrow 0 \rightarrow \ldots $ be the complex where the two G's are in the first and zeroth place and with the map $G \rightarrow G$ the identity map. By [13] $\overline{G}$ is a Gorenstein projective complex; by Proposition 6, we have that $Ext^1(\overline{G}, L)=0$. By \cite{enochs:11:relative2}, Proposition 2.1.3, $Ext^1(G, L_0)=0$.\\
Thus $P_0 \rightarrow M$ is a special Gorenstein projective precover.
\end{proof}

\begin{proposition}
Let $R$ be a right coherent ring. If $R$ is $n$-left perfect then $(\mathcal{GP}, \mathcal{GP}^\bot)$ is a complete hereditary cotorsion pair.
\end{proposition}

\begin{proof}
Let $X \in ^\bot(\mathcal{GP}^\bot)$. By Proposition 6, there exists an exact sequence $ 0 \rightarrow M \rightarrow P \rightarrow X \rightarrow 0$ with $P$ Gorenstein projective and with $M \in \mathcal{GP}^\bot$. But then $Ext^1(X,M) = 0$, so $P \simeq M \oplus X$, and therefore $X$ is Gorenstein projective. So $(\mathcal{GP}, \mathcal{GP}^\bot)$ is a cotorsion pair.\\
By Corollary 2, the pair $(\mathcal{GP}, \mathcal{GP}^\bot)$ is complete.\\
The pair $(\mathcal{GP}, \mathcal{GP}^\bot)$ is hereditary because the class of Gorenstein projective modules is closed under kernels of epimorphisms.
\end{proof}

We recall that if $R$ is a left noetherian ring such that $i.d._R R \le n$ then by \cite{enochs:00:relative}, Proposition 9.1.2, $R$ is left $n$-perfect.

By the above we obtain:\\

\begin{proposition}
Let $R$ be  a right coherent and left noetherian ring such that $i.d._RR \le n$. Then:\\
1) the class of Gorenstein projective modules is special precovering in $R-Mod$;\\
2) $(\mathcal{GP}, \mathcal{GP}^\bot)$ is a complete hereditary cotorsion pair in $R-Mod$.\\
3) the class of Gorenstein projective complexes is special precovering in $Ch(R)$.
\end{proposition}

Sergio Estrada\\
University of Murcia,
Murcia, Spain\\
Email: sestrada@um.es

Alina Iacob\\
Georgia Southern University,
Statesboro, GA 30460-8093\\
Email:  aiacob@GeorgiaSouthern.edu

Sinem Odabasi\\
University of Murcia,
Murcia, Spain\\
Email: sinem.odabasi1@um.es

\end{document}